\documentclass[a4paper]{article}
\usepackage[T2A]{fontenc} 
\usepackage[utf8]{inputenc} 
\usepackage[english]{babel} 
\usepackage{alphabeta} 
\usepackage{authblk}
\usepackage[frozencache=true,cachedir=minted-cache]{minted}
\usepackage{amsmath, amsthm, amsfonts}
\usepackage{enumerate}
\usepackage{hyperref}
\usepackage{amssymb}
\usepackage{graphicx}
\usepackage[all]{xy}
\usepackage{orcidlink}

\newtheorem{thm}{Theorem}[section]
\newtheorem{cor}[thm]{Corollary}
\newtheorem{lem}[thm]{Lemma}
\newtheorem{prop}[thm]{Proposition}
\newtheorem{exam}[thm]{Example}
\theoremstyle{definition}
\newtheorem{defi}[thm]{Definition}
\theoremstyle{remark}
\newtheorem{nota}[thm]{Notation}
\newtheorem{rem}[thm]{Remark}

\providecommand{\keywords}[1]
{

	\small	
	\textbf{\textit{Keywords:}} #1
}

\providecommand{\AMSMSC}[1]
{

	\small	
	\textbf{\textit{AMS MSC:}} #1
}

\begin{document}

\title{The Global Fibered Representation Ring}

\author[1]{J. Miguel Calderón \orcidlink{0009-0007-4685-8837}}
\author[2]{Alberto G. Raggi-Cárdenas \orcidlink{0000-0003-1720-1733}}

\affil[1]{Centro de Investigación en Matemáticas, A.C., Unidad Mérida,\\
Parque Científico y Tecnológico de Yucatán,\\
Carretera Sierra Papacal--Chuburná Puerto Km 5.5,\\
Sierra Papacal, Mérida, Yucatán 97302, México.\\
\texttt{calderonl@cimat.mx}
}

\affil[2]{Centro de Ciencias Matemáticas,\\
Universidad Nacional Autónoma de México,\\
Morelia, Michoacán 58089, México.\\
\texttt{agraggi@gmail.com}
}

\maketitle
\thispagestyle{plain} 

\begin{center}
\end{center}

\begin{center}
\end{center}
\begin{abstract}
In this paper, we combine the notions of the fibered Burnside ring and 
the character ring, viewing both as fibered biset functors, to define 
what we call the global fibered representation ring of a finite group. 
We compute all ring homomorphisms from this ring to the complex numbers, 
describe its spectrum and its connected components, and determine the 
primitive idempotents of the ring after tensoring with $\mathbb{Q}$, 
together with their conductors.

\end{abstract}
\AMSMSC{16Y99, 18D99, 20J15.}

\section{Introduction}
Let $A$ be an abelian group and let $G$ be a finite group. An $A$-fibered $G$-set is a left $A \times G$-set $X$ that is free as an $A$-set and has finitely many $A$-orbits. A morphism between two $A$-fibered $G$-sets is an $A \times G$-equivariant map. The $A$-fibered $G$-sets together with their morphisms form a category, which we denote by $_Gset^A$.

The $A$-fibered Burnside ring of $G$, denoted by $B^A(G)$, is defined as the Grothendieck group of $_Gset^A$, taken with respect to disjoint union (see Section 1.7 of \cite{fibered}). The category $_Gset^A$ admits coproducts and is symmetric monoidal with respect to a tensor product $\otimes_A$ (see Section 2.1 of \cite{fibered}), and these structures induce the ring structure on $B^A(G)$. The $A$-fibered Burnside ring was introduced in greater generality by Dress \cite{dressRing}. 

For any commutative ring $R$ with unity, we set 
$B^A_R(G) := R \otimes_{\mathbb{Z}} B^A(G)$.  
Similarly, one defines $A$-fibered $(G,H)$-bisets and the Grothendieck group 
$B^A(G,H) := B^A(G \times H)$.  
Fibered bisets admit a tensor product construction (see Section 2.1 of \cite{fibered}), giving rise to the $A$-fibered biset category over $R$, denoted by $R\mathcal{C}^A$. Its objects are finite groups, and for two objects $G$ and $H$, the morphisms are 
$B^A_R(G,H) := R \otimes_{\mathbb{Z}} B^A(G,H)$,  
with composition given by the tensor product of $A$-fibered bisets.

The $R$-linear covariant functors from $R\mathcal{C}^A$ to $_R\mathbf{Mod}$, together with natural transformations, form the abelian category $\mathcal{F}^A_R$ of $A$-fibered biset functors over $R$ (see \cite{fibered}).

The $A$-fibered Burnside ring was first introduced in greater generality by Dress in~\cite{dressRing}.  For any commutative ring  $R$, we set   $B^A_R(G):= R \otimes_{\mathbb{Z}} B^A(G)$. We write  $\mathcal{M}^A(G) = \{(H, \phi) \mid H \leq G, \phi \in \operatorname{Hom}(H, A)\}$ on which \(G\) acts by conjugation, and denote by \([\mathcal{M}^A(G)]\) a set of representatives of the $G$-conjugacy classes in \(\mathcal{M}^A(G)\). The elements of \([\mathcal{M}^A(G)]\) parametrize a \(\mathbb{Z}\)-basis of \(B^A(G)\).
Similarly, one defines $A$-fibered $(G,H)$-bisets and their Grothendieck groups $B^A(G,H)$, defined as $B^A(G \times H)$. 
Moreover, fibered bisets admit a tensor product construction giving rise to the $A$-fibered biset category over a commutative ring $R$, denoted by $R\mathcal{C}^A$. 
Its objects are again finite groups, and its morphism sets are given by $B^A_R(G, H):= R \otimes B^A(G, H)$. The $R$-linear  covariant  functors from $R\mathcal{C}^A$ to $_R\mathbf{Mod}$, together with natural transformations, form the abelian category $\mathcal{F}^A_R$ of $A$-fibered biset functors over $R$ (see~\cite{fibered}).
We write $\mathcal{M}^A(G) = \{(H,\phi) \mid H \leq G,\ \phi \in \operatorname{Hom}(H,A)\}$,
on which $G$ acts by conjugation. We denote by $N_G(H,\phi)$ the stabilizer of $(H,\phi)$, and  we denote by $[\mathcal{M}^A(G)]$ a set of representatives of the $G$-conjugacy classes. 
\begin{nota}
Let $G$ be a finite group. For any $(H,\phi) \in \mathcal{M}^A(G)$, we set 
\begin{align*}
\left[\frac{G\times A}{\{ (h, \phi(h)) \mid h\in H\}}\right] := [H,\phi]_G.
\end{align*}
Let  $X$ be   an $A$-fibered $G$-set and $x \in X$, we denote by $G_x \leq G$ the stabilizer of the $A$-orbit of $x$, and by $\phi_x: G_x \longrightarrow A$ the map defined by the equation $gx = \phi_x(g)x$ for any $g \in G_x$.  
We have,   the $G \times A$-orbit of $x$ lies in the class $[G_x, \phi_x]_G$, and the set 
$\{[H, \phi]_G \mid (H, \phi) \in [\mathcal{M}^A(G)] \}$ is a $\mathbb{Z}$-basis of $B^A(G)$.
\end{nota}

One example of these functors is the Character ring
\begin{exam}
    Let $A = \mathbb{C}^\times$. For a finite group $G$, we denote by $R_{\mathbb{C}}(G)$  the character ring of $\mathbb{C}[G]$-modules. In \cite{fibered}, Boltje and  Coşkun give the character ring an $A$-fibered biset functor structure,  $R_{\mathbb{C}}^{\mathbb{C}^\times}$,  in the following way. The assignment  $G \mapsto R_\mathbb{C}(G)$ gives rise to a $\mathbb{C}^\times$-fibered biset
functor $R_{\mathbb{C}^\times}$ by mapping the standard basis element 
$\left[U, \phi\right]_{G\times H}$ 
of $B^{\mathbb{C}^\times}(G, H)$ to the map:
$
R_\mathbb{C}(H) \to R_\mathbb{C}(G), \quad [M] \mapsto \left[\mathrm{Ind}^{G \times H}_{U}(\mathbb{C}^\phi) \otimes_{\mathbb{C}H} M \right], 
$
where $\mathbb{C}^\phi$ denotes the one-dimensional $\mathbb{C}U$-module associated with the homomorphism $\phi:U \longrightarrow \mathbb{C}^\times$.
\end{exam}
In Section \ref{Def global} we will define the global fibered representation ring, which is inspired by the global representation ring introduced in \cite{raggiglobal}. This construction is analogous to previous ones, such as Witherspoon’s ring of $G$-vector bundles \cite{sara} and Nakaoka’s ring \cite{nakaoka2008structure}. In Section \ref{marks} we describe the morphisms from the global fibered   representation ring to $\mathbb{C}$, which will play a key role in Section \ref{prime ideals} for characterizing the prime ideals of the global fibered representation ring.
On the other hand, in Section \ref{idempotents sec} we compute the idempotents of the global fibered  representation ring tensored with $\mathbb{Q}$, and in Section \ref{conect component} we determine the connected components associated to these idempotents within the global fibered representation ring. This analysis allows us to provide applications to other areas of representation theory, including an equivalence related to the Feit–Thompson theorem. Finally, in Section \ref{fubtor seccion} we endow the global fibered  representation ring with the structure of a functor of fibered bisets.

\section{The global fibered representation ring } \label{Def global}
Now, we define the global fibered representation ring.  It is important to recall the global
representation ring is defined in  \cite{raggiglobal}, and this ring has an $A$-fibered biset functor structure (see  \cite{karleyTesis} or \cite{calderonlower}).   The next functor is inspired by the global representation ring functor.  By analogy, it plays a role similar to that between the  Burnside functor and the  Burnside fibered functor. 
Throughout this paper, we denote by $A$ an abelian group.
\begin{defi}
Given  a group $G$ and  a $G$-set $X$, a $\mathbb{C}G$-module $V$ is said to be $X$-graded if
\begin{align*}
    V=\bigoplus_{x\in X} V_x,
\end{align*}
where each $V_x$ is a $\mathbb{C}$-subspace such that  such that $g V_x= V_{g x}$ for all $g\in G$ and $x\in X$.
\end{defi}
\begin{defi}
Let $G$ be a finite group. We denote by $\aleph_G$ the category whose objects are pairs $(X, V)$, where $X$ is an $A$-fibered $G$-set and $V$ is an $X/A$-graded $\mathbb{C}G$-module, with $X/A$ denoting the set of all $A$-orbits of $X$. 

The morphisms in $\aleph_G$ from $(X, V)$ to $(Y, W)$ are pairs $(\alpha, f)$, where $\alpha: X \longrightarrow Y$ is a morphism of $A$-fibered $G$-sets, $f: V \longrightarrow W$ is a morphism of $\mathbb{C}G$-modules, and $f(V_x) \subseteq W_{\alpha(x)}$ for all $x \in X$. The composition in $\aleph_G$ is defined naturally, and the identity morphism of $(X, V)$ is $(\mathrm{Id}_X, \mathrm{Id}_V)$, where $\mathrm{Id}_X$ is the identity function on the $A$-fibered $G$-set $X$, and $\mathrm{Id}_V$ is the identity map of the $\mathbb{C}G$-module $V$.We say that the element $(X, V)$ is isomorphic to $(Y, W)$ if there exists a morphism $(\alpha, f)$ such that $\alpha$ is a bijection and $f$ is an isomorphism of $\mathbb{C}G$-modules. We denote by $[X, V]$ the isomorphism class of $(X, V)$ in $\aleph_G$.
\end{defi}
The category  $\aleph_G$   has a coproduct defined by  $$(X, V)\oplus (Y, W)=(X\sqcup Y, V\oplus W),$$ where $\sqcup$ is the disjoint union of the sets and  $\oplus$ is the direct sum of modules. 
We defined  $T^A(G):=G_0(\aleph_G, \sqcup)$  the Grothendieck group.
\begin{defi}
We define the global $A$-fibered  representation ring of the group $G$ as the quotient
\begin{align*}
\text{Д} ^A (G) = \dfrac{T^A(G)}{ \left\langle [X,V\oplus W]-[X,V]-[X,W]\right\rangle_\mathbb{Z} }.
\end{align*}
We also write $[X, V ]_G$ to denote elements of $\text{Д}^A(G)$.  This group has a ring structure, where the product is defined by  $[X, V]_G\cdot[Y,W]_G:=[X\otimes_A Y, V\otimes W]_G$. Moreover, by definition of addition in $\text{Д}^A(G)$, the set $\{ [X, S]\in \text{Д}^A(G) \mid X \text{ is transitive and } S \text{ simple } \mathbb{C}G-\text{module}  \}$ forms a $\mathbb{Z}$-basis of $\text{Д}^A(G)$. 
\end{defi}

\begin{defi}
    Let $H$ be a subgroup of $G$, $\phi \in Hom(H,A)$, and let  $W$ be  a $\mathbb{C}H$-module. We denote  by   $[H,\phi, W ]_G$  the isomorphism class of the object $ ([G/H, \phi]_G, i^G_H (W ))$, where  $i_H^G(W ):=\mathbb{C}G \otimes_{\mathbb{C}H} W$ is the induced module equipped with the canonical $G/H$-grading.  
Note that $[H, \phi, W]_G = [K, \psi, U]_G$ if and only if there exists  $g \in G$ such that 
$
{}^g H = K, \quad {}^g \phi = \psi, \quad \text{and} \quad {}^g W \cong U 
$ 
as $\mathbb{C}H$-modules. These elements are the generating elements of $\text{Д}^A(G)$.
\end{defi}
The next lemma describes how the product of two generating elements in $\text{Д}^A(G)$ is computed
\begin{lem}[Mackey]
   Let $ [H,\phi,V]_G$ and $ [K, \psi, W]_G$ be elements of  $\text{Д}^A(G)$, We have the following product formula:

    \begin{align*}
        [H,\phi,V]_G \cdot [K, \psi, W]_G= \sum_{x\in [H\backslash G/K]} [H\cap {^xK}, \phi\cdot {^x \psi}, res^H_{H\cap {^xK}} V \otimes_{\mathbb{C}(H\cap {^xK})} res^{^xK}_{H\cap {^xK}} {^x W} ]_G
    \end{align*}  
\end{lem}
\begin{proof}
\begin{align*}
    [H,\phi,V]_G \cdot [K, \psi, W]_G&=[[H,\phi]_G,i^G_H(V)]_G \cdot [[K, \psi]_G, i^G_K(W)]_G\\
    &=[[H,\phi]_G\otimes_A [K, \psi]_G,i^G_H(V) \otimes i^G_K(W)]_G\\
    &= \sum_{x\in [H\backslash G/K]} [H\cap {^xK}, \phi\cdot {^x \psi}, res^H_{H\cap {^xK}} V \otimes_{\mathbb{C}(H\cap {^xK})} res^{^xK}_{H\cap {^xK}} {^x W} ]_G.
\end{align*}
The last equality follows from Mackey's formula for $A$-fibered bisets 
(see Corollary 2.5 of \cite{fibered}). 
\end{proof}
Let $G$ be a finite group and let $[H,\phi,S]_G \in \text{Д}^A(G)$, where $S$ is a simple $\mathbb{C}G$-module.  
Recall that the underlying $A$-fibered $G$-set of $[H,\phi]_G$ has $A$-orbits indexed by $G/H$. Thus, $S$ decomposes as  $S=\bigoplus_{x\in [G/ H ]} S_x $,  where each $S_x$ is a $\mathbb{C}H$-module. Since $S$ is simple as a $\mathbb{C}G$-module, it follows that all $S_x$ are simple $\mathbb{C}H$-modules. Hence,
$
S \;\cong\; \mathbb{C}G \otimes_{\mathbb{C}H} S_y,
\quad \text{for any } y \in G/H.$  Moreover, the set
$$
\Big\{\, [H,\phi,S]_G \in \text{Д}^A(G)
\;\Big|\;
(H,\phi) \in [\mathcal{M}^A(G)]
\text{ and } S \text{ is a simple } \mathbb{C}H\text{-module}
\Big\}
$$
forms a $\mathbb{Z}$-basis of $\text{Д}^A(G)$. The group $G$ acts on this basis by conjugation,\- 
${}^g[H,\phi,S]_G := [\,{}^gH,\;{}^g\phi,\;{}^gS\,]_G,$
where ${}^g\phi({}^gh) := \phi(h)$, and ${}^gS$ denotes the $\mathbb{C}({}^gH)$-module obtained by transporting the action via conjugation.
A natural question arises about when the ring has a finite basis.\\
Note that the number of simple $\mathbb{C}G$-modules is finite for any finite group $G$. Moreover, the cardinality of the basis of $B^A(G)$ is given by
$
\sum_{H \in [S_G]} \bigl|\overline{\operatorname{Hom}(H,A)}\bigr|,
$
where $\overline{\operatorname{Hom}(H,A)}$ denotes the set of orbits of $\operatorname{Hom}(H,A)$ under the action of $N_G(H)$ (see Remark 1.10 of \cite{romero_essentialAlgebra}).  

Thus, the basis of $\text{Д}^A(G)$ is finite if and only if 
$\sum_{H \in [S_G]} \bigl|\overline{\operatorname{Hom}(H,A)}\bigr| < \infty,$
which is equivalent to requiring that 
$ \bigl|\operatorname{Tor}_{\exp(G)}(A)\bigr| < \infty$.

\section{Ring homomorphisms from $\text{Д}^A(G)$ to $\mathbb{C}$} \label{marks}

The objective of this section is to describe all ring homomorphisms from $\text{Д}^A(G)$ to $\mathbb{C}$. To this end, we introduce the following set.

\begin{defi}
Let $G$ be a finite group. We set
$
\mathcal{T}(G) := \{ (H, \psi, a) \mid (H, \psi) \in \mathcal{M}^A(G)\ \text{and}\ a \in H \}.
$
The group $G$ acts on $\mathcal{T}(G)$ by conjugation. We denote by $[\mathcal{T}(G)]$ a set of representatives of its $G$-conjugacy classes.  
Finally, for $(H, \psi, a) \in \mathcal{T}(G)$, we write $N_G(H, \psi, a)$ 
for its stabilizer under this action.
 
\end{defi}

\begin{defi}\label{marca}
 Let $(H,\psi, a) \in \mathcal{T}(G)$. We define the species of  $(H,\psi, a)$ by 
\begin{align*}
  S_{H, \psi, a} :\text{Д}^A(G) &\longrightarrow \mathbb{C}\\
   [X,V]_G &\longrightarrow  \sum_{\substack{ x \in [X/A]\\ (H, \psi) \leq {(G_x,\phi_x)} }}  \chi_{V_x} ( {{} a} ).
    \end{align*} 
where $[X/A]$ denotes the set of $A$-orbits of $X$, and $(G_x,\phi_x)$ is the stabilizer pair of the element $x$ and $\chi_{V_x}$  is the character of $V_x$. Note that
\begin{align*}
    S_{H,\psi, a} ([X,V\oplus W])&=  \sum_{\substack{ x \in [X/A]\\ (H, \psi) \leq {(G_x,\phi_x)} }}  \chi_{(V\oplus W)_x} ( {{} a} )=  \sum_{\substack{ x \in [X/A]\\ (H, \psi) \leq {(G_x,\phi_x)} }}  \chi_{V_x} ( {{} a} )+  \chi_{W_x} ( {{} a} )\\
    &=  \sum_{\substack{ x \in [X/A]\\ (H, \psi) \leq {(G_x,\phi_x)} }}  \chi_{V_x} ( {{} a} )+  \sum_{\substack{ x \in [X/A]\\ (H, \psi) \leq {(G_x,\phi_x)} }}  \chi_{W_x} ( {{} a} )\\
    &= S_{H,\psi, a} ([X,V]) + S_{H,\psi, a} ([X, W]),
\end{align*}
Hence $S_{H,\psi,a}$ is well defined.\\
One can prove directly from the definition that these maps are ring homomorphisms. Indeed, 
$ S_{H,,\psi, a}(X \sqcup Y,\, V \oplus W)= S_{H,,\psi, a}(X, V) + S_{H,,\psi, a}(Y, W),$ and
\begin{align*}
S_{H,,\psi, a}(X \otimes Y,\, V \otimes W)
&= \sum_{(x,y)\in [(X\otimes Y)/A] \atop (H,\psi) \leq (G_{x\otimes y},\phi_{x\otimes y}) } \chi_{ V_{x}}\otimes \chi_{W_{y}}(a)\\
&= \sum_{x\in [X/A] \atop (H,\psi, \leq (G_x,\phi_x) } \sum_{y\in [Y/A] \atop (H,\psi)\leq (G_y,\phi_y) } \chi_{  V_{x}(a)}\cdot \chi_{W_{y}}(a)\\
&= \left(\sum_{x\in [X/A] \atop (H,\psi)\leq (G_x,\phi_x)} \chi_{V_{x}}(a)\right)
   \left(\sum_{y\in [Y/A] \atop (H,\psi)\leq (G_y,\phi_y)} \chi_{W_{y}}(a)\right) \\
&= S_{H,,\psi, a}(X, V)\cdot S_{H,,\psi, a}(Y, W).
\end{align*}
In particular, for a basic element $[K,\phi,S]_G$, we have
     \begin{align*}
   S_{L, \psi, a}( [K, \phi, S]_G)=  \sum_{\substack{ x \in [G/K]\\ (L. \psi) \leq {^x(H,\phi)} }}  \chi_S ( {{} ^xa} ).
    \end{align*}
\end{defi}
Note that, by definition, $S_{L,\psi,a} = S_{L,\psi,{}^g a}$ for any $g \in N_G(K,\psi)$.  
\begin{rem}
Let $(L,\phi)$ and $(K,\psi)$ be  elements of $\mathcal{M}^A(G)$. Following \cite{garciaSpecies}, García defines
    \begin{align*}
   \gamma_{(L,\phi)(K,\psi)}:=|\lbrace xK\in G/K \mid {^x(}L,\phi)\leq {^x}  (K,\psi)  \rbrace|. 
\end{align*}
\end{rem}

\begin{lem} Let $(H,\phi)$ and $(K,\psi)$ be elements of $\mathcal{M}^A(G)$  and let $S$ be a simple $\mathbb{C}K$-module. Then
\begin{align*}
       S_{H,\phi,1}([K,\psi,S]_G)=dim(S) \cdot \gamma^G_{(H,\phi), (K,\psi)},\\   
       S_{H,\phi,1}([K,\psi,\mathbb C]_G)= \gamma^G_{(H,\phi), (K,\psi)}.
\end{align*}
 
\end{lem}
\begin{proof}
    \begin{align*}
         S_{H,\phi,1}([K,\psi,S]_G)&=\sum_{\substack{x\in [G/K] \\ (H,\phi) \leq ^x(K,\psi)}}\chi_{{}^xS} (1)\\
       &=dim(S)\cdot |\lbrace x\in [G/K]  \mid  (H,\phi) \leq  {{}^x(K,\psi)}|\\
       &= dim(S) \cdot \gamma^G_{(H,\phi), (K,\psi)}.
    \end{align*}
In particular, 
$S_{H,\phi,1}([K,\psi,\mathbb C]_G)= \gamma^G_{(H,\phi), (K,\psi)}$.
\end{proof}
Let $(H,\phi) \in \mathcal{M}^A(G)$, and let $N = N_G(H,\phi)$. Let $N$ act on $\mathbb CH$ and $R(H)$ by conjugation, and consider the fixed point subrings $\mathbb{C}H^{\,N}$ and $R(H)^{\,N}$. For each $b \in H$, let $\overline{b}$ denote the sum of all $N$-conjugates of $b$ (that is, the orbit of $b$ under the conjugation action of $N$). 
Note that
$
\overline{b} \;=\; \sum_{x \in [N / C_N(b)]} x b x^{-1}.
$
For each irreducible character $\chi_S$ of $H$, let $\overline{\chi_S}$ denote the sum of all its $N$-conjugates.
It is known that the set
$
\{\overline{b} \mid b \in H\}
$
forms a $\mathbb{C}$-basis of $\mathbb{C}H^{\,N}$, and that the set
$
\{\overline{\chi_S} \mid S \text{ a simple } C_H\text{-module}\}
$
forms a $\mathbb{C}$-basis of $R(H)^{N}$.
\begin{thm}[Theorem 15 of  \cite{raggiglobal}]\label{forma bilineal}
Consider the bilinear form
\begin{align*}
\langle \ , \ \rangle \colon (\mathbb{C}H)^N \times R(H)^N \longrightarrow \mathbb{C}
\end{align*}
defined on the basis elements by
$ 
\langle \overline{b}, \overline{\chi_S} \rangle = 
\chi_S\!\left( \sum_{x \in [N/H]} bx \right).
$
Then this bilinear form is well defined (it does not depend on the choice of representatives of 
$b$, $x$, or $\chi_S$), and it is non-degenerate in both variables.  
In particular, the dimensions of $(\mathbb{C}H)^N$ and $R(H)^N$ coincide 
\end{thm}

\begin{lem}
Let $(H,\phi,a)$ and $(K,\psi,b)$ be elements in $\mathcal{T}(G)$. Then
    \begin{align*}
        S_{H,\phi,a}= S_{K, \psi, b}  \txt{  if  and only if  }  \ (H,\phi,a)=_G (K, \psi, b)
    \end{align*}
\end{lem}
\begin{proof}
$\Rightarrow ]$This follows directly from the definition of the morphism (see Definition~\ref{marca}).\\
$\Leftarrow ]$
    First, note that
    \begin{align*}
        S_{H,\phi,a}([K,\psi, \mathbb C]_G)= S_{K, \psi, b}(([K,\psi ,\mathbb C]_G) = \gamma^G_{(K,\psi), (K,\psi)} \neq 0.
    \end{align*}
Hence, $(H,\phi) \leq_G (K,\psi)$. Similarly, by symmetry, $(K,\psi) \leq_G (H,\phi)$, so there exists $g \in G$ such that $ (H,\phi) = {^g(K,\psi)}$.
Without loss of generality, we may assume $(H,\phi) = (K,\psi)$.
It remains to show that $a$ and $b$ are conjugate in $N_G(H,\phi)$. For every simple $\mathbb{C}H$-module $S$, we have
    \begin{align*}
         S_{H,\phi,a} ([H,\phi, S]_G) =\langle \overline{a}, \overline{S}\rangle = S_{H,\phi,b} ([H,\phi, S]_G) =\langle \overline{b}, \overline{S}\rangle 
    \end{align*}
By Remark~\ref{forma bilineal}, the bilinear form is non-degenerate, so we obtain 
$a =_{N_G(H,\phi)} b$. Therefore, $S_{H,\phi,a} = S_{K,\psi,b}$.
\end{proof}
By the previous lemma, it follows that the sets $\{S_{H,\phi,a} \mid (H,\phi,a) \in [\mathcal{T}(G)]\}$, are all distinct marks
\begin{lem}
    We have that, 
    \begin{align*} 
        \bigcap_{(H, \phi, a) \in [\mathcal{T}(G)]} Ker(S_{H, \phi, a})=0
    \end{align*}
\end{lem}
\begin{proof}
Consider the function
\[
S : \text{Д}^A(G)  \longrightarrow \prod_{[H,\phi,b]\in [\mathcal{T}(G)]} \mathbb{C}
\]
that sends $[K,\psi, ,S]$ to $S_{H,\phi, ,a}([K,\psi, S])$ in the appropriate coordinate. Notice that the kernel of $S$ is the intersection of the kernels of the maps $S_{H,\phi, b}$. Since the ideal $I$ of $\text{Д}^A(G) $ is contained in the kernel of each $S_{H,b}$,  and the matrix representing $S$ with respect to the classes $[K,\psi ,S]$ has entries $S_{H,\phi, b}([K,\psi,S])$.
If this matrix were invertible, then $S$ would be injective. The matrix has blocks of the form $S_{H, \phi, b}([H,\phi, S])$ along the diagonal, and zeros below these blocks. Thus, it suffices to prove that each of these blocks is an invertible square matrix.
Notice that, by Theorem~15 of \cite{raggiglobal},
\begin{align*}    
S_{H, \phi, b}([H,\phi,S]) & =\sum_{\substack{x\in [G/H] \\ (H,\phi) \leq {^x(H,\phi)}}}\chi_{{}^xS}(a)=\sum_{\substack{x \in [G/H] \\ H \leq { ^xH}}} \chi_S(ax)\\
&= \sum_{x \in N/H} \chi_S(ax),
\end{align*}
which is the matrix of a non-degenerate bilinear form. Hence, it is an invertible square matrix.
\end{proof}

\begin{defi}
The square matrix whose rows are indexed by $[H, \phi, b]\in [\mathcal{T}(G)]$ and columns by $[K, \psi, S]_G$, 
and whose entries are  $S_{H,\phi,b}([K,\psi,S]_G)$.
is called the \emph{table of species}. This matrix consists of diagonal blocks 
(indexed by the conjugacy classes of elements in $[\mathcal{M}^A(G)]$), 
with zeros outside these blocks. The last block on the diagonal corresponds 
to the character table of $G$. 
\end{defi}
Bby Theorem~19 of \cite{raggiglobal}, it is invertible. Consequently, we obtain the following result.
\begin{cor}
    The maps $S_{H, \phi,b}$  are precisely all  the ring homomorphisms from $\text{Д}^A (G)$ to $\mathbb C$.
\end{cor}

\begin{cor}
The canonical map
$$
\prod_{[H,\phi,b] \in [\mathcal{T}(G)]} 
S_{H,\phi,b} \colon 
\text{Д}^A(G) 
\longrightarrow 
\prod_{[H,\phi,b] \in [\mathcal{T}(G)]} \mathbb{C}
$$
is an injective ring homomorphism.
\end{cor} 
\section{The prime  ideals of $\text{Д}^A(G)$} \label{prime ideals}
Throughout this section, let $n$ be the order of the group $G$, and 
$\omega$ a primitive complex $n$-th root of unity. 
For each $S_{H,\phi,b}$, let $P_{H,\phi,b}$ denote its kernel, 
and let $\operatorname{Im}_{H,\phi,b}$ denote its image, which lies inside $\mathbb{Z}[\omega]$. 
Let $\mathfrak{m}$ be a maximal ideal of $\mathbb{Z}[\omega]$. 
We define $S^{\mathfrak{m}}_{H,\phi,b}$ as the composition of 
$S_{H,\phi,b}$ (restricted to $\mathbb{Z}[\omega]$) 
followed by the quotient map to $\mathbb{Z}[\omega]/\mathfrak{m}$, 
which is a finite field. 
Let $P^{\mathfrak{m}}_{H,\phi,b}$ be the kernel of 
$S^{\mathfrak{m}}_{H,\phi,b}$, 
and let $p$ denote the characteristic of $\mathbb{Z}[\omega]/\mathfrak{m}$.
\begin{lem}\label{product b}
Let $H$ be a subgroup of $G$, and let $b, y \in H$ with $by = yb$, where $y$ has order $p^k$. Then 
$S_{H,\phi,b}(z) - S_{H,\phi,by}(z) \in \mathfrak{m}$
for all $z \in \text{Д}(G)$ and all maximal ideals $\mathfrak{m}$ of $\mathbb{Z}[\omega]$ 
of characteristic $p$ (that is, containing $p\mathbb{Z}[\omega]$).
\end{lem}

\begin{proof}
Let $z = [K,\psi,S]_G$. Then
\begin{align*}
(S_{H,\phi,b} - S_{H,\phi,by})(z)
  &= \sum_{\substack{x \in [G/K]\\ (H,\phi) \leq {}^x(K,\psi)}} 
     \big( \chi_S(b^x) - \chi_S((by)^x) \big).
\end{align*}
Since $b^x$ and $y^x$ commute, they are simultaneously diagonalizable on $S$. 
Hence there exists a basis $\{v_1, \ldots, v_r\}$ of $S$ such that 
$b^x v_i = u_i v_i$ and $y^x v_i = t_i v_i$, with $t_i^{p^k} = 1$. 
Therefore,
\begin{align*}
\chi_S(b^x) - \chi_S((by)^x)
  &= \sum_{i=1}^r u_i (1 - t_i).
\end{align*}
Since $(1 - t_i)^{p^k} \equiv 0 \pmod{p}$, it follows that $(1 - t_i) \in \mathfrak{m}$. 
Hence $(S_{H,\phi,b} - S_{H,\phi,by})(z) \in \mathfrak{m}$.
\end{proof}
\begin{cor}
Let $(H, \phi, b) \in \mathcal{T}(G)$, and let $y \in H$ commute with $b$. 
Then 
$P_{H,\phi,b}^\mathfrak{m} = P_{H,\phi,by}^\mathfrak{m}. $
\end{cor}

\begin{proof}
   Let $z \in \text{Д}^A(G)$. 
By Lemma~\ref{product b}, we have 
$
z \in \ker(S_{H,\phi,b}^\mathfrak{m}) 
\quad \text{if and only if} \quad 
z \in \ker(S_{H,\phi,by}^\mathfrak{m}).
$
Hence, the corresponding prime ideals coincide.
\end{proof}

\begin{thm}
    The spectrum of $\text{Д}^A(G)$ consists of all ideals of the form $P_{H,\phi,  b}$ and $P_{H,\phi, b}^\mathfrak{m}$.
\end{thm}
\begin{proof}
    Let $P\leq \text{Д}^A(G)$ be an ideal  prime. Since  $\bigcap _{(H,\phi b)\in [\mathcal{T}(G)]} P_{H,\phi,  b} =0 <P$, and  $Hom(G,A)$ is a finite set, there exists some $(H,\phi,b)$ such that 
$P_{H,\phi,b} \subseteq P$.   Then  $P/P_{H,\phi, b}$ is an ideal of $\text{Д} ^A (G)/P_{H,\phi, b} \cong Im_{H,\phi, b} \leq \mathbb {Z}[w]$, Hence, $P / P_{H,\phi,b}$ is either $0$ (in which case $P = P_{H,\phi,b}$),  or a maximal ideal $m$ of $\mathbb{Z}[\omega]$.   In the latter case, by the Going-up Theorem, we have
 $P=m \cap Im_{H,\phi,b}$, that is, $P=P_{H,\phi,b}^\mathfrak{m}$.
\end{proof}
\begin{cor}
Let $P \leq \text{Д}^A(G)$ be a maximal ideal such that $P_{H,\phi,b} \subseteq P$. 
Then there exists a maximal ideal $\mathfrak{m}$ of $\mathbb{Z}[\omega]$ such that $ P = P_{H,\phi,b}^\mathfrak{m}.$ 
\end{cor}
\begin{lem} \label{b,1 espectro}
 The ideals   $P_{H,\phi,b}$ and $P_{H,\phi,1}$ lie  in the same connected component of the spectrum of $\text{Д}^A(G)$.
\end{lem}
\begin{proof}
    First, we note that  $P_{H,\phi,b}$ and $P_{H,\phi,b}^\mathfrak{m}$ are in the same connected component.
    Let $|b| = p_1^{k_1} \cdots p_r^{k_r}$, and write $b = b_1 \cdots b_r$ with $|b_i| = p_i^{k_i}$. 
For each prime $p_i$, there exists a maximal ideal $m_i \subset \mathbb{Z}[\omega]$ of characteristic $p_i$. 
By Lemma~\ref{product b}, the ideals 
     $P_{H,\phi,b}^{m_1}$ and $P_{H,\phi,b_2\cdots b_r}^{m_1}$ are in the same connected component. Then $P_{H,\phi,b}$ and $P_{H,\phi,b_2\cdots b_r}$  are in the same connected component, similarly, we obtain that $P_{H,\phi,b}$ $P_{H,\phi,b_3\cdots b_r}$, ... and  $P_{H,\phi,1}$ are in the same connected component.
    
\end{proof}
\begin{lem} \label{1 1 espctro}
    Let $L \trianglelefteq K$ be of index $p$ and let $\phi \in Hom(K,A)$. Then, the ideals $P_{L,\phi,1}$ and $P_{K,\phi,1}$ lie in the same connected component of the spectrum of $\text{Д}^A(G) $.
\end{lem}
\begin{proof} Let $[K,\psi, S]_G\in \text{Д}^A(G) $, we has 
 \begin{align*}
(S_{L,\phi,1}-S_{H,\phi,1} )([K,\psi, S]_G)=( \gamma_{(L,\phi)(K,\psi)}-\gamma_{(H,\phi)(K,\psi)})\cdot dim S
 \end{align*}
where 
\begin{align*}
   \gamma_{(L,\phi)(K,\psi)}=|\lbrace x\in G/H \mid {^x(}L,\phi)\leq {^x}  (K,\psi)  \rbrace|. 
\end{align*}
We note that $K$   act on $\lbrace x\in G/H \mid (L,\phi)\leq {^x}  (K,\psi)  \rbrace$, and define 
\begin{align*}
Y=\lbrace x\in G/H \mid (L,\phi) \leq {^x}(K,\psi) \text{ and } (H,\phi) \lneq {^x}(K,\psi) \rbrace.
\end{align*}
 $Y$  is a union of orbits no-trivial $K/L$-orbits. Then $(S_{L,\phi,1}-S_{H,\phi,1} )([K,\psi, S])$ is a multiple of $p$, i.e., it is an element to $m$. 
\end{proof}
\begin{cor}
Let \(L \trianglelefteq K\) be a normal subgroup of index \(p\). 
Then, for every maximal ideal \(\mathfrak{m}\) of \(\mathbb{Z}[w]\), the ideals 
\(P_{L,\phi,1}^{\mathfrak{m}}\) and \(P_{K,\phi,1}^{\mathfrak{m}}\) lie in the same connected component of \(\text{Д}^A(G)\).
\end{cor}

\begin{defi}
Let $K$ be a group.  
We define $O^s(K)$ , called the solvable residual of \(K\), to be 
the smallest normal subgroup of \(K\) such that the quotient $K / O^s(K)$  is solvable.
\end{defi}

The subgroups $O^s(K)$ influence the structure of the rings associated with groups, 
as shown in \cite{raggi1990burnside} and \cite{yoshida1983idempotents}.

\begin{cor} \label{conect component}
The ideals $P_{K,\phi,1}$ and $P_{O^s(K),\phi,1}$ belong to the same connected component 
of the spectrum.
\end{cor}

\begin{lem}
    Let $ P_{K, \phi,b}$ be contained in a prime ideal $P$ of $\text{Д}^A(G)$. Then either  $P=P_{K, \phi,b}$ or there exists a maximal ideal $m$ of $\mathbb Z[w]$ such that $P = P_{K, \phi,b}^\mathfrak{m}$ .
\end{lem}

\begin{proof}
    Assume that $P\neq P_{K, \phi,b}$. Then the  quotient  $P/P_{K,\phi,b}$ embeds as an ideal of $Im_{K.\phi,b}$ in $\mathbb Z[w]$. Since $Hom (K, A)$ is a finite set, By the
Going-up Theorem, there exists a maximal ideal $m$ of $\mathbb Z[w]$  that intersects $Im_{K.\phi,b}$ in this embedded copy of $P/P_{K,\phi,b}$. Thus,  $P=P_{K,\phi,b}^\mathfrak{m}$
\end{proof}
Now we can describe the connected components of the spectrum of $\text{Д}^A(G)$.
\begin{thm}
    Let $H$  be a perfect subgroup of $G$, that is, $H = O^s(H)$. For a given \(\phi \in \mathrm{Hom}(H,A)\), let \(X_{H,\phi}\) denote the set of all ideals of the forms $P_{K,\psi, b}$ and  $P_{K,\psi, b}^\mathfrak{m}$, such that  $O^s(K)={^gH}$ and $\psi={^g\phi}$ for some $g\in G$. Then $X_{H,\phi}$ is a connected component of the spectrum of $\text{Д}^A(G)$, and every  connected component is one of the $X_{H,\phi}$  with $H$ perfect. 
\end{thm}
\begin{proof}
    Let $P_{K,\psi,1} \in X_{H,\phi}$. By the Corollary \ref{conect component}, we have $P_{K,\psi,1}$ and $P_{H,\phi,1}$ lie in the same connected component. Now by the Lemma \ref{b,1 espectro}, we have  $P_{K,\psi,b}$ and  $P_{K,\psi,b}^\mathfrak{m}$  belongs to the same connected component of  $P_{H,\phi,1}$. Therefore, $X_{H,\phi}$ is contained in a connected component of the spectrum 
of $\text{Д}^A(G)$.
Now we prove that $X_{H,\phi}$ is both closed and open. First, we have
\begin{align}
    X_{H,\phi} \subseteq 
    \bigcup_{(O^s(K),\psi)\,=_G\,(H,\phi)} V(P_{K,\psi,b}),
\end{align}
where 
$
V(P_{K,\psi,b}) = 
\{\, P \in \mathrm{Spec}(\text{Д}^A(G)) \mid P_{K,\psi,b} \leq P \,\}.
$
The reverse inclusion follows from the fact that the only prime ideals containing 
$P_{K,\psi,b}$ must be of the form $P_{K,\psi,b}^\mathfrak{m}$ for some maximal ideal $m$ of $\mathbb{Z}[w]$.  Hence, $V(P_{K,\psi,b})$ is a subset of $X_{H,\phi}$, and thus
\begin{align}
    X_{H,\phi} = 
    \bigcup_{(O^s(K),\psi)\,=_G\,(H,\phi)} V(P_{K,\psi,b}).
\end{align}
Since the sets $V(P_{K,\psi,b})$ are closed, it follows that $X_{H,\phi}$ is closed. 
Moreover, the sets $X_{H,\phi}$ are either equal or disjoint, so that
\begin{align*}
    \mathrm{Spec}(\text{Д}^A(G)) = \bigsqcup X_{H,\phi}.
\end{align*}
Because the number of such sets is finite, each $X_{H,\phi}$ is open. 
Therefore, $\mathrm{Spec}(\text{Д}^A(G))$ is a finite disjoint union of these sets.

\end{proof}

  \begin{cor}
Let $(K,\phi,a)$ and $(L,\alpha,c)$ be elements of $\mathcal{T}(G)$, and let $\mathfrak{m}$ and $\mathfrak{n}$ be maximal ideals of $\mathbb{Z}[\omega]$. 
Then the ideals $P_{K,\phi,a}^{\mathfrak{m}}$ and $P_{L,\alpha,c}^{\mathfrak{n}}$ lie in the same connected component of $\mathrm{Spec}(\text{Д}^A(G))$ if and only if $  (O^s(K),\phi) =_G (O^s(L),\alpha)$.
\end{cor}

\section{Idempotent of $\mathbb{Q}[\omega]\text{Д}^A(G)$ } \label{idempotents sec}
\begin{defi}
Let     $ (K,\psi) \in \mathcal{M}^A(G)$  and $a\in K$. We define

    \begin{align*}
        e_{K, \psi, a} :=  \frac{1}{|N_G(K,\psi)\cap C_G(a)|}\sum_{\substack{ (\langle a \rangle, \psi) \leq (H, \psi) \leq (K,\psi)\\ S\in Irr(H)}} \mu ((H,\psi),(K,\psi)) \overline{ \chi_S(a)}[H, \psi, S],
    \end{align*}
    where $\mu$ is the Möbius function of the poset $\mathcal{M}^A(G)$. Here, $e_{K, \psi, a} $ is an element of  $\mathbb{Q}[\omega]\text{Д}^A(G)$. 
\end{defi}

 \begin{thm}
Let $(L,\psi), (K,\phi) \in \mathcal{M}^A(G)$, with $a \in L$ and $b \in K$. Then
\begin{align*}
    S_{L, \psi, a}(e_{K, \phi , b}) = 
    \begin{cases}
        1 & \text{if } (L, \psi, a) =_G (K, \phi, b) \\
        0 & \text{otherwise.}
    \end{cases}
\end{align*}
\end{thm}

\begin{proof} Let
\begin{align*}
    m_{K,\phi,b}  =\frac{1}{|N_G(K,\phi)\cap C_G(b)|}.
\end{align*}
Then
\begin{align*}
S_{L, \psi, a}(e_{K, \phi , b})
&= m_{K, \phi, b} \cdot 
\sum_{\substack{(H, \alpha)\in \mathcal{M}^A(G)\\ (\langle a \rangle ,\phi)\leq (H,\alpha) \leq (K,\phi)\\ S\in \mathrm{Irr}(H)\\ x\in [G/H]\\ (L,\psi) \leq {^x(H,\alpha)}}} 
\mu((H,\alpha),(K,\phi)) \, \overline{\chi_S(b)} \, \chi_S({^x a})\\
&= m_{K, \phi, b} \cdot 
\sum_{\substack{(H, \alpha)\in \mathcal{M}^A(G)\\ (\langle a \rangle ,\phi)\leq (H,\alpha) \leq (K,\phi)}} 
\mu((H,\alpha),(K,\phi)) 
\sum_{\substack{ x\in [G/H]\\ (L,\psi) \leq {^x(H,\alpha)}}} 
\sum_{S\in \mathrm{Irr}(H)} \overline{\chi_S(b)} \, \chi_S({^x a}).
\end{align*}
If $a \neq_G b$, then $S_{L, \psi, a}(e_{K, \phi , b}) = 0$. 
Without loss of generality, assume $a = b$. 
If $(L,\psi) \lneq_G (K,\phi)$, then also $S_{L, \psi, a}(e_{K, \phi , b}) = 0$. Now, assume that $(L,\psi) \leq_G (K,\phi)$. We compute
\begin{align*}
    S_{L, \psi, a}(  e_{K, \phi , b})&=  m_{K, \phi, b} \cdot \sum_{\substack{(H, \alpha)\in \mathcal{M}^A(G)\\ (\langle a \rangle ,\phi)\leq (H,\alpha) \leq (K,\phi)}} \mu ((H,\alpha),(K,\phi)) \frac{1}{|H|} \sum_{\substack{ x\in G \\ a=_H {^xa}\\ (L,\psi) \leq ^x(H,\alpha) }}  \sum_{S\in Irr(H)}  \overline{ \chi_S(a)} \chi_S(^xa)\\
&=  m_{K, \phi, b} \cdot \sum_{\substack{(H, \alpha)\in \mathcal{M}^A(G)\\ (\langle a \rangle ,\phi)\leq (H,\alpha) \leq (K,\phi)}} \mu ((H,\alpha),(K,\phi)) \frac{1}{|H|} \sum_{\substack{ x\in G \\ a=_H {^xa}\\ (L,\psi) \leq ^x(H,\alpha) }}  |C_H(a)|\\
&=  m_{K, \phi, b} \cdot \sum_{\substack{(H, \alpha)\in \mathcal{M}^A(G)\\ (\langle a \rangle ,\phi)\leq (H,\alpha) \leq (K,\phi)}}\frac{|C_H(a)|}{|H|} \sum_{\substack{ x\in G \\ a=_H {^xa} }} \zeta((L,\psi)^x, (H, \alpha)) \cdot\mu ((H,\alpha),(K,\phi)).
\end{align*}
Let $\{a_1, \dots, a_r\}$ be the orbit of $a$ under conjugation by $H$, with $r = |H|/|C_H(a)|$ and $a_i = a^{h_i}$. Then the previous sum becomes

\begin{align*}
   S_{L, \psi, a}(  e_{K, \phi , b})&=  m_{K, \phi, b} \cdot \sum_{\substack{ (H,\alpha  ) \in \mathcal{M}^A(G)}}\frac{|C_H(a)|}{|H|} \sum_{i=1}^r\sum_{\substack{ x\in G \\ ^xa= a_i }} \zeta((L,\psi)^x, (H, \alpha)) \cdot\mu ((H,\alpha),(K,\phi)).
\end{align*}
Note that $a^x = a_i = a^{h_i}$  if and only if $a^{xh_i} = a$, and if $L^x \leq H$, then $L^{xh_i} \leq H$, so
the change of variable $x = x^{h_i}$  and the fact that $r = |H|/|C_H (a)| $  yield
\begin{align*}
  S_{L, \psi, a}(  e_{K, \phi , b})&=  m_{K, \phi, b} \cdot \sum_{\substack{ (H,\alpha  ) \in \mathcal{M}^A(G)}} \sum_{\substack{ x\in G \\ ^xa= a }} \zeta((L,\psi)^x, (H, \alpha)) \cdot\mu ((H,\alpha),(K,\phi))\\
  &=   m_{K, \phi, b} \cdot \sum_{\substack{ x\in C_G(a) }} \sum_{\substack{ (H,\alpha  ) \in \mathcal{M}^A(G)}}  \zeta((L,\psi)^x, (H, \alpha)) \cdot\mu ((H,\alpha),(K,\phi))\\
  &=0, 
\end{align*}
y the Möbius inversion formula, this sum is $0$ if there is no $x \in C_G(a)$ such that ${^x(L,\psi)} = (K,\phi)$.  
Finally, if $(L,\psi,  a) = (K, \alpha, b)$, then  
\begin{align*}
 S_{L, \psi, a}(  e_{K, \phi , b})&=   m_{K, \phi, b} \cdot \sum_{\substack{ x\in C_G(a) }} \sum_{\substack{ (H,\alpha  ) \in \mathcal{M}^A(G)}}  \zeta((L,\psi)^x, (H, \alpha)) \cdot\mu ((H,\alpha),(K,\phi))\\
 &=   m_{K, \phi, b} \cdot \frac{1}{|N_G(L,\psi)\cap C_G(a)|}\\
 &=1.
\end{align*}
\end{proof}
Then,  the set of elements  $\{e_{K,\phi, b}\mid (K,\phi, b) \in [\mathcal{T}(G)]\}$ are  idempotent orthogonal primitive of $\mathbb{Q}[\omega]\text{Д}^A(G)$.

\section{Conductors of idempotents } \label{conducotors}

\begin{rem}
    Let $ (K,\psi,a) \in \mathcal{T}(G) $. Then
$$ N_G(K,\psi,a) = \{\, g \in G \mid ({}^g K, {}^g \psi, {}^g a) = (K,\psi,a) \,\} = N_G(K,\psi) \cap C_G(a). $$ 
\end{rem}
\begin{prop}\label{S cong 0}
Let $\omega$ denote a primitive $|G|$-th root of unity, and let 
$$
\text{Д}^A_\omega(G) := \mathbb{Z}[\omega] \otimes \text{Д}^A(G)
$$ 
be the ring obtained from $\text{Д}^A(G)$ by extending scalars to $\mathbb{Z}[\omega]$. Consider the map
$$
\prod_{[H,\phi,a] \in [\mathcal{T}(G)]} 
S_{H,\phi,b} \colon 
\text{Д}^A_\omega(G)
\longrightarrow 
\prod_{[H,\phi,a] \in [\mathcal{T}(G)]} \mathbb{Z}[\omega].
$$

Let $\alpha \in \text{Д}^A_\omega(G)$. Then, for all  $[H, \phi, a]\in [\mathcal{T}(G)]$, we have
$$
\sum_{g \in N_G(H,\phi, a)} S_{\langle g, H \rangle, \phi, g}(\alpha) \equiv 0 \pmod{|N_G(H,\phi, a)|},
$$
where $\langle g,H \rangle$ denotes the subgroup of $G$ generated by $g$ and $H$.
\end{prop}

\begin{proof}
It suffices to consider the case where $\alpha$
 is a basis element of the form $[X,V]_G$. 
Let $N := N_G(H,\phi,a)$. Then
    \begin{align*}
        \sum_{g \in N} S_{\langle g, H \rangle,\phi, g}([X,V])&=\sum_{g \in N} \sum_{\substack{x\in X\\  (\langle g,H\rangle ,\phi) \leq (G_x,\phi_x) } }\chi_{V_x}(g)\\
        &=\sum_{x\in X^H}\sum_{g\in N_x} \chi_{V_x} (g) =\sum_{x\in X^H} |N_x|\text{dim} V_x^{N_x}
    \end{align*}
where $N_x := \{\, g \in N : gx = x \,\}$. Now let $n \in N$ and $x \in X^H$. For any $h \in H$, we have  $hnx=n(n^{-1}hn)x=nh^\prime x=nx$, for some $h^\prime \in H$, since $N$ normalizes $H$. Hence $nx \in X^H$, and therefore $N$ acts on $X^H$.
Grouping the terms in the above sum according to the $N$-orbits in $X^H$, we obtain
\begin{align*}
   \sum_{g \in N} S_{\langle g, H \rangle,\phi, g}([X,V])=  |N| \sum_{x\in [N/ X^H]} dim V_x^{N_x},
\end{align*}
which is divisible by $|N|$, as required.
\end{proof}
\begin{thm}
Let $[K,\psi,a]\in[\mathcal{T}(G)]$.
Consider the primitive idempotent $e_{K,\psi,a}$ in 
$ 
\mathbb{Q}\text{Д}^A_\omega(G),
$
corresponding to the class $[K,\psi,a]$.
Then the conductor of $e_{K,\psi,a}$ into the global fibered representation ring  $\text{Д}^A_\omega(G)$ is $|N_G(H,\psi,a)|$,
where $N_G(K,\psi,a)$ denotes the stabilizer of the triple $(K,\psi,a)$ in $G$.
\end{thm}
\begin{proof}

Let $ n $ denote the conductor of $ e_{K,\psi,a} $. Recall that the primitive idempotent $ e_{K,\psi,a} $ is given by the formula,
 
\begin{align*}
    e_{K, \psi, a} =  \frac{1}{|N_G(K,\psi)\cap C_G(a)|}\sum_{\substack{ (\langle a \rangle, \psi) \leq (H, \psi) \leq (K,\psi)\\ S\in Irr(H)}} \mu ((H,\psi),(K,\psi)) \overline{ \chi_S(a)}[H, \psi, S],
\end{align*}
From this formula we immediately see that  $n \leq |N_G (K, \psi, a)|$. Now, applying Proposition   \ref{S cong 0} to the class $[K,\psi,1]$, and since $N_G(K,\psi,1)=N_G(K,\psi)$, we obtain 
\begin{align*}
    m|N_G(K,\psi)|=&\sum_{g\in N_G(K,\psi)} S_{\langle g,K \rangle,\psi,g} (n\cdot e_{K,\psi, a}) =n\cdot|\{ g\in K \mid g=_{N_G(K,\psi)} a \} |\\
    =&n\cdot|N_G(K,\psi)\cdot a|=n \bigg|\frac{N_G(K,\psi)}{C_{N_G(K,\psi)} (a)} \bigg|=\frac{|N_G(K,\psi|)}{|N_G(K,\psi,a)|}.
\end{align*}
Equivalently, we can write $n = m\,|N_G(K,\psi,a)|$, that is, 
$|N_G(K,\psi,a)| \leq n$. This yields the desired equality for the conductor.
\end{proof}

\begin{defi}
    Let $[H,\phi]\in \mathcal{M}^A(G)$, we define 
    \begin{align*}
        e^s_{H,\phi} =\sum_{{\substack{(K,\psi, a ) \in [\mathcal{T}(G)] \\ (O(K), \psi)=_G( H, \phi) \\ a\in [K/G] }}} e_{K,\psi,  a}.
    \end{align*}
    
    \end{defi}
 \begin{prop}
    The elements $e^s_{H,\phi}$, for $[H,\phi]_G \in [\mathcal{M}^A(G)]$, 
are primitive, orthogonal idempotents of $\text{Д}^A(G)$.
 \end{prop}
 \begin{proof} By definition, the elements  $e_{H,\phi}^s$  are  orthogonal idempotent.   Let $e\in \text{Д}_\omega^A(G)$.  Then $e$ is a sum of the primitive idempotents   of $\mathbb{Q}[\omega]\text{Д}^A(G)$.  Since each $S_{H,\phi,a}$ is a ring homomorphism, we have, $S_{H,\phi,a}(e)$ is either $0$ or $1$.  Define 
     \begin{align*}
         \mathbb{E}_e:= \lbrace (H,\phi, a) \in [\mathcal{T}(G)] \mid S_{H,\phi,a}(e)=1 \rbrace. 
     \end{align*}
Then we can write 
\begin{align*}
    e=\sum_{(H,\phi,a) \in \mathbb{E}_e} e_{H,\phi,a}.
\end{align*}
By  Lemma \ref{b,1 espectro} and Lemma \ref{1 1 espctro},   if $(H,\phi, a) \in \mathbb{E}_e$ and there exists  $(K,\psi,b)\in [\mathcal{T}(G)]$  such that $(O(K), \psi)=_G(H, \phi)$,   then $(K,\psi,b) \in \mathbb{E}_e$. Thus, 
\begin{align*}
    e=\sum_{(H,\phi) \in [\mathbb{E}_e]} e^s_{H,\phi}
\end{align*}
where $[\mathbb{E}_e]$ denotes the set of pairs $(H,\phi)$ such that 
$H$ is a perfect subgroup of $G$ and there exists $a \in H$ with 
$(H,\phi,a) \in \mathbb{E}_e$.
 \end{proof}
\begin{cor}
There exists a non-trivial primitive idempotent in $\text{Д}_\omega^A(G)$ 
if and only if $G$ is not solvable.
\end{cor}

\section{The Functor $\text{Д}^A$. } \label{ fubtor seccion}
The ring $\text{Д}^A(G)$ has the structure of an $A$-fibered biset functor, which is described explicitly in the following definition.
\begin{defi}
We define  the global $A$-fibered  representation ring functor  
\begin{align*}
\text{Д}^A :R\mathcal{C}^A &\longrightarrow R\text{-Mod}\\
G&\longmapsto \text{Д}^A(G)
\end{align*}
Let   $G$,   $H$ be    objects of $ R\mathcal{C}^A$, and  let   $U$ be  an  $A$-fibered $(G,H)$-biset, i.e., an element of  $Hom_{R\mathcal{c}} (H, G)$.   We define the map 
\begin{align*}
\text{Д}^A(U): \text{Д}^A(H) &\longrightarrow \text{Д} ^A(G)\\
[X,V]_H&\longmapsto [U\otimes_{AH} X, \mathbb{C}U \otimes_{ \mathbb{C} AH } V]_G,
\end{align*}
where  $A$ acts on  $V$  via  the trivial action.
\end{defi} 
The functor $\text{Д}^A$ is well-defined, since the tensor product of $A$-fibered bisets respects the isomorphism classes of $A$-fibered bisets, and the direct sum of modules respects the isomorphism classes of modules. If $V=\oplus_{x\in X/A} V_x$, then one has  
\begin{align*}
\mathbb{C}U \otimes_{ \mathbb{C} AH } V  \cong \bigoplus_{u\otimes x\in U\otimes_{AH} X} u\otimes V_x. 
\end{align*}
where  $ u\otimes V_x:=\lbrace u\otimes v_x \mid v_x \in V_x\rbrace$ is a $\mathbb{C}$-vector subspace. 
\begin{nota}
   Let $G$ be a finite group, $K$ a subgroup of $G$, and $N$ a normal subgroup of $G$. 
Moreover, denote the inclusion by $i: K \longrightarrow G$ and the canonical projection by $\pi: G \longrightarrow G/N$. 

We define the elementary $A$-fibered biset operations as follows:
\begin{align*}
{}^A\mathrm{Res}^G_K &:= \left[ \frac{K \times G}{\Delta(K), 1} \right], 
& {}^A\mathrm{Ind}^G_K &:= \left[ \frac{G \times K}{\Delta(K), 1} \right], \\
{}^A\mathrm{Def}^G_{G/N} &:= \left[ \frac{G/N \times G}{\,^{(\pi,1)} \Delta(G)} \right], 
& {}^A\mathrm{Inf}^G_{G/N} &:= \left[ \frac{G \times G/N}{\,^{(1,\pi)} \Delta(G)} \right],
\end{align*}
where
$
{}^{(\pi,1)} \Delta(G) := \{ (gN, g) \mid g \in G \}, \qquad
{}^{(1,\pi)} \Delta(G) := \{ (g, gN) \mid g \in G \}.
$

Finally, if $f : H \longrightarrow G$ is an isomorphism, we set
$
\mathrm{iso}(f) := \left[ \frac{G \times H}{ \{ (f(h), h) \mid h \in H \}, 1} \right].
$
\end{nota}
In the following lemma, we illustrate the importance of these elementary $A$-fibered bisets.
\begin{prop}[Proposition 2.8 \cite{fibered}]
Let $(U, \phi) \in \mathcal{M}^A(G \times H)$. Then, there exist subgroups $P \leq G$, $\hat{K} \unlhd P$, $Q \leq H$, and $\hat{L} \unlhd Q$ such that
$$
[U, \phi]_{G\times H}
\;\; \cong \;\;
{}^A\mathrm{Ind}^G_P \;\; \otimes_A \;\; {}^A\mathrm{Inf}^P_{P/\hat{K}} \;\; \otimes_A \;\; X \;\; \otimes_A \;\; {}^A\mathrm{Def}^Q_{Q/\hat{L}} \;\; \otimes_A \;\; {}^A\mathrm{Res}^Q_H,
$$
where $X$ is a suitable $A$-fibered $(P/\hat{K}, Q/\hat{L})$-biset.
\end{prop}

A natural question is how this functor behaves with elementary operations,
\begin{prop}
Let $G$ be a finite group. Then the action of the elementary $A$-fibered biset operations on elements of $\text{Д}^A$ is given as follows:

\begin{itemize}
    \item Let $H \leq G$ and let $[K,\phi,S]_G \in \text{Д}^A(G)$. Then
    $$
        \text{Д}^A({}^A\mathrm{Res}^G_H)([K,\phi,S]_G) = 
        \sum_{x \in [H \backslash G / K]} 
        [H \cap {}^x K, {}^x\phi, \mathrm{Res}^{^x K}_{H \cap {}^x K} ({}^x S)]_H.
    $$
    
    \item Let $H \leq G$ and let $[K,\phi,S]_H \in \text{Д}^A(H)$. Then
    $$
        \text{Д}^A({}^A\mathrm{Ind}^G_H)([K,\phi,S]_H) = [K,\phi,S]_G.
    $$
    
    \item Let $N \trianglelefteq G$ and let $[K/N, \phi, S]_{G/N} \in \text{Д}^A(G/N)$. Then
    $$
        \text{Д}^A({}^A\mathrm{Inf}^G_{G/N})([K/N, \phi, S]_{G/N}) = [K, \phi, \mathrm{Inf}^K_{K/N}(S)]_G.
    $$
    
    \item Let $N \trianglelefteq G$ and let $[K,\phi,S]_G \in \text{Д}^A(G)$. Then
  
{\scriptsize
\[
\text{Д}^A({}^A\mathrm{Def}^G_{G/N})([K,\phi,S]_G) = 
\begin{cases}
    [KN/N, \overline{\phi}, 
    \mathrm{Iso}^{KN/N}_{K/K \cap N} \, 
    \mathrm{Def}^K_{K/K \cap N} S]_{G/N} 
    & \text{if } \phi = 1 \text{ on } K \cap N, \\[6pt]
    0 & \text{otherwise}.
\end{cases}
\]
}

\end{itemize}

\end{prop}
\begin{proof}
    The statement follows directly from the definitions. More generally, this property was established on page~48 of \cite{caderonTesis}.
\end{proof}

\section{Applications} \label{aplications}

\begin{thm}
The order of $G$ is even if and only if there exists a ring homomorphism 
$ 
f : D(G) \longrightarrow \mathbb{C}
$ 
whose image lies in the real numbers, and such that $f$ is not a fibered mark (that is, $f \neq S_{H,\phi,1}$ for all subgroups $H$). 
\end{thm}
\begin{proof}
Assume that there exists an element $b \in G$ of order $2$. Then the ring homomorphism $S_{G,1, b}$ is not a fibered mark, and for all $\mathbb{C}G$-modules $V$ we have $S_{G,\phi, b}([G, \phi,  V]) = \chi_V(b)$ is a real number.\\
Now assume that there exists a ring homomorphism $S_{H,\psi, b}$ which is not a  fibered mark (that is, $b \neq 1$) and real-valued, so $S_{H,\phi, ,b} = S_{H,\phi,b^{-1}}$. This means that $b$ is conjugate to $b^{-1}$, say there exists $g \in G$ such that $g b g^{-1} = b^{-1}$. Note that $g^2 b g^{-2} = g b^{-1} g^{-1} = b$. 
If the order of $g$ is even, then so is the order of $G$. If the order of $g$ is an odd number $r$, then  $b = g^r b g^{-r} = b^{-1}$,
and since $b \neq 1$, then the order of $b$ is $2$.
\end{proof}
\begin{thm}
 Let $G$ be an arbitrary finite group. If there exists a non-trivial primitive  idempotent in $ \text{Д}^A(G)$ then there exists a ring homomorphism from $ \text{Д}^A(G)$ to $\mathbb{R}$ which is  not a fibered  mark (that is, it is not of the form $S_{H,\phi,1}$ )
\end{thm}
\begin{proof}    
Assume there is no ring homomorphism from $ \text{Д}^A(G)$ to $\mathbb{R}$ other than the marks. 
This implies that the order of $G$ is odd. By the Feit–Thompson Theorem (see \cite{feit}), the group $G$ is solvable, 
so the only primitive idempotent in $ \text{Д}^A(G)$ is $1$
\end{proof}

\section*{Declarations}
Conflict of interest statement We have no potential conflict of interest.

\bibliographystyle{plain}
\bibliography{biblio1}

\end{document}